\documentclass[reqno,ammssymb, 12pt]{amsart}
\usepackage{txfonts}
\usepackage{mathrsfs}
\usepackage{}
\usepackage{xypic}
\usepackage{graphicx}
\usepackage{amsfonts}
\usepackage{amssymb}
\usepackage{bbm}

\usepackage{color, xcolor}
\usepackage{mathtools}
\usepackage{extarrows}
\usepackage{dsfont}
\usepackage{amscd}
\usepackage{amsthm,amsmath}

\numberwithin{equation}{section}
\def\R{{\mathbb{Z}}}

\def\Z{\mathbb Z}
\def\G{\Gamma}

\def\lr#1{\left\langle #1\right\rangle}
\def\lrb#1{\left(#1\right)}
\def\a{\alpha}

\def\g{\gamma}

\def\kk{\mathds{k}}
\def\Z{{\mathbb Z}}

\def\mq{m_q}

\def\x{{\bold x}}
\def\y{{\bold y}}
\def\e{{\bold e}}
\def\q{{\bold q}}

\def\mb#1{\left[\kern-0.15em\left[ {#1}
 \right]\kern-0.15em\right]}

\def\s#1{s_{#1,q}}
\def\g#1{\sigma_{#1, q}}

\def\siq{s_{i, q}}
\def\sjq{s_{j, q}}

\def\ol{\overline}

\def\uldim{\underline{\bf dim\,}}

\newtheorem{thm}{Theorem}[section]
\newtheorem{lem}[thm]{Lemma}
\newtheorem{prop}[thm]{Proposition}

\newtheorem{defn}[thm]{Definition}
\newtheorem{ex}[thm]{Example}
\newtheorem{remark}[thm]{Remark}

\date{}
\begin{document}

%\sffamily
%\ttfamily

\thispagestyle{empty}

\begin{center}
{\bf{\Large  $q$-Coxeter matrix and $q$-Cartan matrix for a homogeneous bound quiver}
 \footnote{$^*$
% %$^*$ Supported by National Natural
%%Science Foundation of China (Grant No. 11671024)
%
\small Email: slyang@bjut.edu.cn}
}

\bigbreak

\normalsize Shilin Yang$^*$\\
Faculty of Science,
Beijing University of Technology\\
Beijing 100124, P. R.  China
\end{center}

\numberwithin{equation}{section}

\begin{quote}
{\noindent{\bf Abstract.}
The aim of this paper is to introduce the concept of $q$-Coxeter
transformation and $q$-Coxeter matrix for a connected acyclic bound quiver $(Q, I)$
with a homogeneous relations $I$, then to establish the
relationship between $q$-Coxeter matrix and $q$-Cartan matrix of $(Q, I)$.
}
\end{quote}



\section{Introduction}
The definition  of $q$-Cartan matrix, occurred in the earlier literature as the filtered Cartan matrix, for instance see \cite{Ful}. It is reintroduced
for a homogeneous algebra by Bessenrodt and Holm \cite{BT1} in order to obtain new invariants for derived equivalences of gentle algebras.
It is shown that the $q$-Cartan matrix of a skew-gentle algebra is closely related to the $q$-Cartan matrix of the underlying gentle algebra.

The entries of the $q$-Cartan matrix are Poincar\'{e}
polynomials of the graded homomorphism spaces between projective
modules. Loosely speaking, when counting paths in the quiver of the algebra, each
path is weighted by some power of $q$ according to its length. Clearly,
specializing to $q=1$ gives back the usual Cartan matrix $C_A$.
The point of view of $q$-Cartan matrices provides some
new insights into the invariants of the Cartan matrix.

   In the paper \cite{ZZ}, the authors show that $q$-Cartan matrix of
   any self-injective Nakayama algebra is diagonalizable.
 The determinant of $q$-Cartan martrix is computed
   explicitly.

Up to now, there are few works concern the concept of $q$-Coxeter transformation or $q$-Coxeter matrix for
a homogeneous algebra. Motivated by the concept of $q$-Cartan matrix, we shall introduce the concept of $q$-Coxeter
transformation and $q$-Coxeter matrix for a given acyclic quiver $Q$ or homogeneous bound acyclic quiver $(Q, I)$ of
finite global dimension using $q$-\emph{reflections}, the
relationships between its $q$-Coxeter matrix and the $q$-Cartan matrix
are established. More precisely, the $q$-Coxeter transformation is defined by the composition of a series of so-called $q$-reflections $\s{a_i}$ or $\g{i}$ for an admissible numbering $(a_1, \cdots, a_n)$ of $Q$.
The $q$-Coxeter matrix is defined to be the matrix of the corresponding $q$-Coxeter transformation under
a fixed associated basis of $\R^n$. Then we show that the explicit relations  between $q$-Coxeter matrix
and $q$-Cartan matrix, where  $q$-Cartan matrix is referred to the definition of Bessenrodt and
Holm introduced in \cite{BT1}.
In the process one difficulty we overcome is to define
some suitable simple $q$-refections. Fortunately, this can be done
for a homogeneous bound acyclic quiver $(Q, I)$ of finite global dimension, in particular an acyclic quiver $Q$.
It is remarked that
specializing to $q=1$ gives back the usual relations between Coxeter matrix and Cartan matrix.
We hope that this discovery has a good application in
studying representations of a homogeneous  bound acyclic quiver of finite global dimension.

The paper is arranged as follows. In Section 2, we introduce the concept of $q$-reflections  for an acyclic quiver. Some basic properties are established. In Section 3 and Section 4, we introduce the concept of $q$-Coxeter transformations and $q$-Coxeter matrix for an acyclic quiver, the variant relations between $q$-Coxeter matrices and $q$-Cartan matrices are described.  In Section 5, we generalize the concepts and the results
to a homogeneous bound acyclic quiver of finite global dimension.

    In the sequel, we always
    denote $\mb{m, n}=\left\{m, m+1, \cdots, n\right\}$ for $n\ge m$, $\kk$ is an algebraic closed field.

\section{$q$-Cartan matrix for a homogeneous bound quiver}

A finite directed graph $Q$ is called a (finite) quiver. For any arrow $\a$ in $Q$ we
denote by $s(\a)$ its start vertex and by $t(\a)$ its end vertex.
An oriented path $p$ in $Q$
of length $r$ is a sequence $p=\a_1\a_2\cdots\a_r$ of arrows $\a_i$ such that $t(\a_i)=s(\a_{i+1})$
for all $i=1, \cdots, r-1;$ its start vertex is $s(p):=s(\a_1)$ and its end vertex is
$t(p)=t(\a_r)$ (For each vertex $a$ in $Q$, a trivial path $e_a$ of length $0$,
having $a$ as its start and end vertex).

The path algebra $\kk Q$, has as a basis the set of all oriented
paths in $Q$. Multiplication is defined by concatenation of paths: the product of two
paths $p$ and $q$ is defined to be the concatenated path $pq$ if $t(p)=s(q)$, and zero
otherwise.
More general algebras can be obtained by introducing relations on a path algebra.
An ideal $\mathcal{I}\subset \kk Q$ is called admissible if $\mathcal{I}\subseteq J^2$, where $J$ is the
ideal of $\kk Q$ generated by the arrows of $Q$. Any  admissible ideal $\mathcal{I}$ of $\kk Q$ for a finite quiver $Q$ is generated by linear combinations of paths having the same start vertex and the same end vertex.
The pair $(Q, \mathcal{I})$, where $Q$ is a quiver and $\mathcal{I}\subseteq \kk Q$ is an
admissible ideal, is called a quiver with bound relations.
For any bound quiver $(Q, \mathcal{I})$, we can consider the factor algebra
$A=\kk Q/\mathcal{I}$ and identify paths in the quiver $Q$ with their cosets
in $A$. Let $Q_0$ denote the set of vertices of $Q$. For any $a\in Q_0$, we get $1=\sum\limits_{a\in Q_0} e_a$
and
$A=\bigoplus\limits_{a\in Q_0} e_aA$.
The (right)
$A$-modules $P(a):= e_a A$ are the indecomposable projective $A$-modules.
It is known that the category of mod-$A$ of finitely generated right $A$-modules is equivalent to the category rep$_\kk(Q, \mathcal{I})$ of finite dimensional $\kk$-linear representations of $Q_A$ bound by $\mathcal{I}$.

Now, let $(Q, \mathcal{I})$ be a quiver with homogeneous ideal $\mathcal{I}$, that is the relation ideal $\mathcal{I}=\lr{I}$ is generated by the set $I$ of linear combinations of paths having the same length. The path algebra
$\kk Q$ is a graded algebra, with grading given by path lengths. Since $\mathcal{I}$ is homogeneous,
the factor algebra $A=\kk Q/\mathcal{I}$ inherits this grading.The morphism spaces
${\rm Hom}_A(P(j), P(i))\cong e_i Ae_j$ become graded vector spaces.  In this case, $(Q, I)$ is called {\it a  homogeneous bound quiver} and the  graded algebra $A=\kk Q/\mathcal{I}$ is called the {\it homogeneous bound quiver algebra associated to $(Q, I)$}.

\begin{defn}{\rm (see \cite{BT1}).} Let $A=\kk Q/\lr{I}$  be a finite-dimensional homogeneous bound quiver algebra associated to $(Q, I)$. For any vertices $i$ and $j$ in $Q$ let
$$e_iAe_j=\bigoplus_n\lrb{e_iAe_j}_n$$
be the graded components.
The $q$-Cartan matrix $C_q:=C_{Q, I}(q)=(c_{ij}(q))$ of $(Q, I)$ is defined by
the matrix with entries
$$c_{ij}(q):=\sum_n\dim_{\kk}\lrb{e_iAe_j}_n q^n\in \Z[q].$$
In particular, if $I=0$, the corresponding $q$-Cartan matrix
is called
$q$-Cartan matrix of $Q$.
\end{defn}

\begin{ex} Let
$$(Q, I): \xymatrix@=1.5cm{1&2&3\\
\ar@<-.5ex>"1,1";"1,2"_\alpha
\ar@<-.5ex>"1,2";"1,1"_\delta
\ar@<-.5ex>"1,2";"1,3"_\beta
\ar@<-.5ex>"1,3";"1,2"_\gamma
}$$
be the homogeneous bound quiver $(Q, I)$ with $I=\{\alpha\beta,
\ \gamma\delta,\ \delta\alpha-\beta\gamma\}$.
The $q$-Cartan matrix of~$(Q, I)$ has the form
$$C_q = \mbox{{\small $\left( \begin{array}{ccc}
1+q^2 & q & 0 \\ q & 1+q^2 & q \\ 0 & q & 1+q^2
\end{array} \right)$}}.
$$
\end{ex}

For the study,  we always assume that $(Q, I)$ is a finite connected homogeneous bound quiver with
$|Q_0|=n$
without loss of generality.
As is seen, $A=\kk Q/\lr{I}$  is a basic and connected finite dimensional graded $\kk$-algebra with an identity, having
${\rm Rad}(A)/\lr{I}$ as radical and $\{e_a| a\in Q_0\}$ as complete set of primitive orthogonal idempotents.

A connected  homogeneous bound quiver $(Q, I)$ is said to be of finite global dimension if the corresponding basic connected  homogeneous bound quiver algebra $\kk Q/\lr{I}$ is of  finite global dimension. The notation  ${\rm gldim}(Q, I)$ means the global dimension ${\rm gldim}(\kk Q/\lr{I})$ of the algebra $\kk Q/\lr{I}$.

Let $a\in Q_0$, we denote by $S(a)$ the representation $(S(a)_b, \varphi_\a)$ of $Q$ defined as
$$S(a)_b=
\left\{
  \begin{array}{ll}
    0, & \hbox{ if } b\ne a; \\
    \kk, & \hbox{ If } b=a,
  \end{array}
\right.
$$
and
$$\varphi_\a=0 \qquad \hbox{ for all } \a\in Q_1.$$
Let $P(a)=e_aA,$ where $a\in Q_0$ and $A=\kk Q/\lr{I}$.

The following results are well-known.

\begin{lem}\label{lem2-1}
\begin{enumerate}
  \item For any $a\in Q_0$, $S(a)$ viewed as an  $A$-module is isomorphic to  the top
  of the indecomposable projective $A$-module $P(a).$

  \item The set $\{S(a)| a\in Q_0\}$  is a complete set of representations of the isomorphism classes of the
  simple $A$-modules.

  \item If $P(a)=(P(a)_b, \varphi_\beta)$, then $P(a)_b$ is the $\kk$-vector space with basis
the set of all the $\ol{w}=w+I$, with $w$  a path from $a$ to $b$ and for an arrow $\beta: b\to c$, the $\kk$-linear
map $\varphi_\beta: P(a)_b\to P(a)_c$ is given by
the right multiplication by $\ol{\beta}=\beta+I$.
  \item If $I(a)=(I(a)_b, \varphi_\beta)$, then $I(a)_b$ is the dual of $\kk$-vector space with basis
the set of all the $\ol{w}=w+I$, with $w$  a path from $b$ to $a$ and for an arrow $\beta: b\to c$, the $\kk$-linear
map $\varphi_\beta: I(a)_b\to I(a)_c$ is given by
the left multiplication by $\ol{\beta}=\beta+I$.
\end{enumerate}
\end{lem}

We define
$$\uldim_q S(i):=\uldim S(i)\ \hbox{ \big(the  dimension vector of } S(i) \big ), $$
and
$$\uldim_q P(i)=\left(\sum_k\left(\dim (P(i)e_1\right)_k q^k, \cdots, \sum_k\left(\dim (P(i)e_n\right)_k q^k \right)^T,$$

The definition of $\uldim_q I(i)$  is similar.

By Lemma \ref{lem2-1}, one sees that $\dim (P(i)e_j)_k=\dim (e_iAe_j)_k$ is the number of $k$-length paths from $i$ to $j$,
and
$$P(i)=\bigoplus_{k\ge 0} P(i)_k$$ as a $\kk$-vector space, where $P(i)_k$ is the $\kk$-space
spanned by $k$-length sub-paths of all paths which spans $P(i)$.

\begin{prop} \label{prop2-6}
Let $C_q$  be the $q$-Cartan matrix of a connected  homogeneous bound quiver $(Q, I)$.  Then
\begin{enumerate}
  \item The $i$-th row of $C_q$ is $\left(\uldim_q P(i)\right)^T.$
  \item The $i$-th column of $C_q$ is $\uldim_q I(i).$
  \item $\uldim_q P(i)= C_q^T\cdot\uldim_q S(i)$.
  \item $\uldim_q I(i)=C_q\cdot\uldim_q S(i).$
\end{enumerate}
\end{prop}
\begin{proof}
The statement (1) follows from the definition and the obvious equality $e_iAe_j=P(i)e_j$.
The statement (2) follows from the definition and the obvious equality $\sum\limits_k \dim\left(I(i)e_j\right)_k q^k=
\sum \limits_k \dim\left(e_jAe_i\right)_kq^k=c_{ji}(q)$ for all $i, j\in Q_0$.
The left statements are easily from (1) and (2) and the facts that the vectors
$\uldim_q S(1), \cdots, \uldim_q S(n)$
form the standard basis of the free abelian group $\mathbb{Z}^n$, where $n=|Q_0|$.
\end{proof}

\begin{prop} \label{prop2-7} Let $C_q$  be the $q$-Cartan matrix of a connected  homogeneous bound quiver $(Q, I)$ of finite global dimension. Then
$\det C_q\in \{1, -1\}$. In particular,
$$C_q\in GL(n, \mathbb{Z}[q])=\left\{ A\in M_n(\mathbb{Z}[q])|\det(A)\in\{1, -1\}\right\}.$$
\end{prop}
\begin{proof}
Let $n=|Q_0|$ and $a\in Q_0$. By the hypothesis, the simple module $S(a)$ has a projective resolution
$$0\to P_{m_a}\overset{p_{m_a}}{\longrightarrow} \cdots\overset{p_1}{\longrightarrow} P_1\overset{p_0}{\longrightarrow}P_0\overset{\varepsilon}{\longrightarrow} S(a)\to 0$$
in mod-$A$, where $m_a$ is finite. It follows that
$$0\to (P_{m_a})_0\overset{{p_{m_a}|}_0}{\longrightarrow}\cdots \overset{{p_1|}_0}{\longrightarrow} (P_1)_0\overset{{p_0|}_0}{\longrightarrow} (P_0)_0\overset{\varepsilon}{\longrightarrow}  S(a)\to 0$$
and
$$0\to (P_{m_a})_k\overset{{p_{m_a}|}_k}{\longrightarrow}\cdots \overset{{p_1|}_k}{\longrightarrow} (P_1)_k\overset{{p_0|}_k}{\longrightarrow} (P_0)_k\to 0,$$
 are also resolution. %as $\kk$-vector space.
Here $(P_i)_k\ (k\ge 1)$ is the $\kk$-vector space
spanned by $\kk$-paths modulo the relations $I$. Hence
$$\uldim S(i)=\uldim_q S(i)=\sum_{j=0}^{m_a} (-1)^j \uldim_q P_j.$$

By the unique decomposable theorem, each of the modules $P_j$ is the
direct sum of finitely many copies of the modules $P(1), \cdots, P(n)$. Therefore the $a$-th standard basis vector $\uldim S(a)$ of $\mathbb{Z}^n$ is a linear combination of the vector $\uldim_q P(1), \cdots, \uldim_q P(n)$.

It follows that there exists $B\in M_n(\mathbb{Z}[q])$ such that
$$E=\left(
      \begin{array}{c}
        \left(\uldim_q S(1)\right)^T\\
        \vdots \\
        \left(\uldim_q S(n)\right)^T \\
      \end{array}
    \right)=B\left(
              \begin{array}{c}
                \left(\uldim_q P(1)\right)^T \\
                \vdots \\
                \left(\uldim_q P(n)\right)^T \\
              \end{array}
            \right),
$$
where $E$ is the identity matrix.

Consequently, $E=BC_q$ and the result follows.
\end{proof}

By Proposition \ref{prop2-7}, we see that the inverse of $C_q$ belongs to $GL_n(\Z[q])$.

\section{$q$-reflections for an acyclic quiver}

Let  $Q=(Q_0, Q_1)$ be an acyclic quiver associated to
the underlying graph  $\G=(\G_0, \G_1)$, where
$\G_0$ is the set of vertices of $\G$ and
$\G_1$ is the set of edges of $\G$. We  write
$\G_0=\{1, \cdots, n\}\subseteq \mathbb{N}$ if $|\G_0|=n$.
Then $Q$ is said  an orientation of $\G$.
Vertices, which are joined by edges, are called neighbours.
Let $\G_0(i)$ be the set of  neighbours of $i$, and
$a_{ij}=a_{ji}$ the number of edges in $\G$ between
 $i$ and $j$.

 Let
$$\R^\G:=\R^n=\left\{ \x=(x_i)\ |\ x_i\in \R, i\in \G\right\},$$
$q$ a parameter and
$\mq=q^2a_{ij}a_{ji}$ for $i\ne j$.

Let $\e_1, \e_2, \cdots, \e_n$ be a basis of $\R^\G$ and
$\x\in\R^\G$, which is equivalent to $\x=\sum\limits_{k\in\G_0}x_k\e_k\in\R^\G$.

For each $i\in \G_0$, we define
a $\R$-linear map $\siq: \R^n\to \lrb{\R[q]}^n$ as follows:
$$\s{i}(\x)=\x+q\sum_{j\in\G_0(i)}a_{ji}x_j\e_i-2x_i\e_i.$$
The map $\s{i}$ is said to be a $q$-reflection at  $\e_i$.
\begin{remark}\rm
If $q=1$, the definition of $\siq$
coincides with the
simple reflection $s_i$ at $\e_i$.
\end{remark}

It is easy to see that
$$\s{i}\lrb{\e_j}=\left\{
                    \begin{array}{ll}
                      -\e_i, & \hbox{ if } j=i; \\
                      \e_j+qa_{ji}\e_i, & \hbox{ if } j \hbox { is a neighbor  of }  i;  \\
                      \e_j, & \hbox{ otherwise.}
                    \end{array}
                  \right.
$$
\begin{prop} \label{prop3-2} The following statements hold.
\begin{enumerate}
  \item $\s{i}^2=1$;
  \item if $i$ and $j$ are not neighbours, then
  $\s{i}\s{j}=\s{j}\s{i};$
  \item if $i$ and $j$ are neighbours, then
$\siq\sjq\siq-\sjq\siq\sjq=
\left(\mq-1\right)\lrb{\siq-\sjq}.$
\end{enumerate}
\end{prop}
\begin{proof}  Let $\x=\sum\limits_{k\in\G_0}x_k\e_k\in\R^\G$.

(1)  Firstly, we have
\begin{eqnarray*}
\siq^2(\x)&=&\siq\left(\x+q\sum_{j\in\G_0(i)}a_{ji}x_j \e_i-2x_i\e_i\right)\\
&=&\x+q\sum_{j\in\G_0(i)}a_{ji}x_j\e_i-2x_i\e_i
-q\sum_{j\in\G_0(i)}a_{ji}x_j
\e_i+2x_i\e_i\\
&=&\x.
\end{eqnarray*}
Hence $\siq^2=1$.

(2) Assume that $i$ and $j$ are not neighbours. Then
$\siq(\e_j)=\e_j$ and $\s{j}(\e_i)=\e_i$.
Therefore,
\begin{eqnarray*}
\siq\s{j}(\x)&=&\siq\left(\x+q\sum_{k\in\G_0(j)}a_{kj}x_k\e_j
-2x_j\e_j\right)\\
&=&\x+q\sum_{k'\in\G_0(i)}a_{k'i}x_{k'}\e_i-2x_i\e_i+
q\sum_{k\in\G_0(j)}a_{kj}x_k\e_j-2x_j\e_j
\end{eqnarray*}
and
\begin{eqnarray*}
\s{j}\siq(\x)&=&
\s{j}\left(\x+q\sum_{k'\in\G_0(i)}a_{k'i}x_{k'}\e_i-2x_i\e_i
\right)\\
&=&\x+q\left(\sum_{k\in\G_0(j)}a_{kj}x_k\right)\e_j-2x_j\e_j
+q\left(\sum_{k'\in\G_0(i)}a_{k'i}x_{k'}\right)\e_i-2x_i\e_i
\end{eqnarray*}
Hence for all $\x=\sum_{k\in\G_0}x_k\e_k\in\R^\G$, we have
$$\siq\s{j}(\x)=\s{j}\siq(\x)$$
and $\siq\s{j}=\s{j}\siq$.

 (3) If $j$ and $i$ are neighbours, then
\begin{eqnarray*}
\sjq\siq(\x)&=&
\sjq\left(\x+q\left(\sum_{k\in \G_0(i)} a_{ki}x_k\right)\e_i-2x_i\e_i\right) \\
&=&\sjq(\x)
+q\left(\sum_{k\in \G_0(i)} a_{ki}x_k\right)\left(\e_i+qa_{ij}\e_j\right)-
2x_i \e_i-2qa_{ij}x_i \e_j\\
&=&\siq(\x)+\sjq(\x)-\x+q^2\left(\sum_{k\in \G_0(i)}a_{ki}x_k\right)a_{ij}\e_j
-2qa_{ij}x_i \e_j,
\end{eqnarray*}
and
\begin{eqnarray*}
&&\siq\sjq\siq(\x)=
\siq\sjq\left(\x+q\left(\sum_{k\in \G_0(i)} a_{ki}x_k\right)\e_i-2x_i\e_i\right) \\
&=&\siq\left(\x+q\left(\sum_{k'\in \G_0(j)} a_{k'j}x_{k'}\right)\e_j-2x_j\e_j
+q\left(\sum_{k\in \G_0(i)} a_{ki}x_k\right)\left(\e_i+qa_{ij}\e_j\right)-
2x_i \e_i-2qa_{ij}x_i \e_j\right)\\
&=&\x+q\left(\sum_{k\in \G_0(i)} a_{ki}x_k\right)\e_i-2x_i\e_i
-q\left(\sum_{k\in \G_0(i)} a_{ki}x_k\right)\e_i+2x_i\e_i
+q\left(\sum_{k'\in \G_0(j)} a_{k'j}x_{k'}\right)\left(\e_j+qa_{ji}\e_i\right)\\
&&-2x_j\e_j-2qa_{ji}x_j\e_i+q^2\left(\sum_{k\in \G_0(i)}
a_{ki}x_k\right)a_{ij}\left(\e_j+qa_{ji}\e_i\right)
-2qa_{ij}x_i\left(\e_j+qa_{ji}\e_i\right)\\
&=&\x+q^2\left(\sum_{k'\in \G_0(j)} a_{k'j}x_{k'}+
q\left(\sum_{k\in \G_0(i)} a_{ki}x_k\right)a_{ij}
-2a_{ij}x_i\right)a_{ji}\e_i-2qa_{ji}x_j\e_i\\
&&+q\left(\sum_{k'\in \G_0(j)} a_{k'j}x_{k'}
+q\left(\sum_{k\in \G_0(i)} a_{ki}x_k\right)a_{ij}
\right) \e_j-2qa_{ij}x_i\e_j-2x_j\e_j.
\end{eqnarray*}
Hence
\begin{equation}\label{eqn1}
  \siq\sjq\siq(\x)=\mq\left(\siq(\x)-\x\right)+\sjq(\x)+(\ast),
\end{equation}
where
$$(\ast)=
q^2\left(\sum_{k\in \G_0(i)} a_{ki}x_{k}\right)a_{ij}\e_j
+q^2\left(\sum_{k\in \G_0(j)} a_{kj}x_{k}\right)a_{ji}\e_i\\
-2qa_{ji}x_j\e_i-2qa_{ij}x_i\e_j.
$$
In a similar way, we have
\begin{equation}\label{eqn2-2}
\sjq\siq\sjq(\x)
=\mq\left(\sjq(\x)-\x\right)+\siq(\x)+(\ast).
\end{equation}
Thus, if $i$ and $j$ are neighbours, (\ref{eqn1}) and (\ref{eqn2-2}) imply that
\begin{equation*}
\lrb{\siq\sjq\siq-\sjq\siq\sjq}(\x)=\left(\mq-1\right)\lrb{\siq-\sjq}\lrb{\x},
\end{equation*}
for all $\x\in \R^\Gamma$.
Hence
$$\siq\sjq\siq-\sjq\siq\sjq=\left(\mq-1\right)\lrb{\siq-\sjq}.$$
The proof is finished.
\end{proof}

We define the symmetric bilinear form $(-\,,  -)$ of an
orientation
$Q=(Q_0, Q_1)$ of $\G$ as
$$\lrb{\x, \y}_q=\sum_{i\in Q_0}x_iy_i-\frac{q}{2}\sum_{\a\in Q_1}
\lrb{x_{s(\a)}y_{t(\a)}+x_{t(\a)}y_{s(\a)}}\in\Z[q].$$
for $\x, \y\in\R^\G$ with
associated quadratic form
$$
\q_q(\x)=\sum_{i\in\G} x_i^2-q\sum_{\underset{i}\bullet-\underset{j}\bullet}
a_{ij}x_ix_j.
$$
\begin{prop}
$\lrb{\s{k}(\x), \s{k}(\y)}_q=\lrb{\x, \y}_q$
for all
$\x=\sum\limits_{i\in I} x_i\e_i$ and $\y=\sum\limits_{i\in I} y_i\e_i$.
\end{prop}
\begin{proof}
It is noted that
$$\lrb{\e_i, \e_k}_q=-\frac{1}{2}qa_{ik}, \quad \lrb{\e_i, \e_i}_q=1.$$
\begin{enumerate}
  \item If  $j$ and $k$ are neighbours of $i$, then
  \begin{eqnarray*}
\lrb{\siq(\e_j), \siq(\e_k)}_q&=&\lrb{\e_j+qa_{ji}\e_i, \e_k+qa_{ki}\e_i}_q\\
&=&\lrb{\e_j, \e_k}_q+
qa_{ji}\lrb{\e_i, \e_k}_q+qa_{ki}\lrb{\e_j, \e_i}_q+q^2a_{ji}a_{ki}\lrb{\e_i,\e_i}_q\\
&=&\lrb{\e_j, \e_k}_q-\frac{1}{2}q^2a_{ji}a_{ki}-\frac{1}{2}q^2a_{ki}a_{ji}
+q^2a_{ji}a_{ki}\\
&=&\lrb{\e_j, \e_k}_q.
\end{eqnarray*}
  \item If  $j$ are neighbours of $i$ but $k$ is not, then
  \begin{eqnarray*}
\lrb{\siq(\e_j), \siq(\e_k)}_q&=&\lrb{\e_j+qa_{ji}\e_i, \e_k}_q
=\lrb{\e_j, \e_k}_q=\lrb{\e_j, \e_k}_q
\end{eqnarray*}
 If  $k$ are neighbours of $i$ but $j$ is not, the proof is similar.
  \item   If  $j, k$ are all not neighbours of $i$, then $\lrb{\siq(\e_j), \siq(\e_k)}_q=\lrb{\e_j, \e_k}_q.$
  \item If $j=i$ and $k$ is neighbours of $i$ (or if $k=i$ and $j$ is neighbours of $i$), then
    $$\lrb{\siq(\e_j), \ \siq(\e_k)}_q=-\lrb{\e_i, \e_k+a_{ki}\e_i}_q=
   \frac{1}{2}a_{ki}-a_{ki}=-\frac{1}{2}a_{ki}
    =\lrb{\e_i, \e_k}_q.$$
    \item If $j=k=i$, then
    $\lrb{\siq(\e_j), \ \siq(\e_k)}_q=\lrb{-\e_i,  -\e_i}_q=\lrb{\e_j,  \e_k}_q.$
\end{enumerate}
The proof is finished.
\end{proof}

\begin{remark} \rm By Proposition \ref{prop3-2}, we can define a $\kk$-algebra ${\mathcal A}_p$, which is
generated by $x_i, i\in \mb{1,n}$ with the following relations
\begin{eqnarray*}
   &&(a) \quad x_i^2=1;  \\
   &&(b) \quad x_ix_j=x_jx_i, \hbox{ if } \ |i-j|\ge 2;\\
  &&(c) \quad  x_ix_jx_i-x_jx_ix_j=\lrb{p-1}\lrb{x_i-x_j}, \hbox{ if }\
  |i-j|=1.
\end{eqnarray*}
However, this algebra is not new. To see this, we set
$$ \bold{i}=\sqrt{-1}, \quad \theta=\sqrt{p-1},  \quad
q=\frac{1+\theta\bold{i}}{1-\theta\bold{i}}.
$$
The parameter $q$ is well defined since $\theta\ne -\bold{i}$.
As a consequence, we see that
$A_p\cong H_{q}(n)$,  where
$H_q(n)$ is the $q$-Hecke algebra
referred to the definition in \cite[p. 158]{PM}.
In other word, we get a realization of $q$-Hecke algebra.
\end{remark}

\section{$q$-Coxeter transformation and $q$-Coxeter matrix for an acyclic quiver}
Given an acyclic orientation $Q=(Q_1, Q_0)$ of $\G$,
there is a bijection between $Q_0$ and the set $\mb{1, n}$ such that if we have
an arrow $j\to i$, then $i<j$: Now, let $a_1$ be any sink in $Q$, then consider the full
subquiver $Q(a_1)$ of $Q$ having as set of points $Q_0-\{a_1\}$; let $a_2$ be a sink of $Q(a_1)$, and continue by induction. Such a numbering of the points of $Q_0$ is called an admissible numbering of $Q$.

Let $(a_1, \cdots, a_n)$ be an admissible numbering of  $Q$ and
let $E=\R^n$ and $F=\lrb{\R[q]}^n$. Set
$$c_q=\s{a_1}\circ\cdots \circ\s{a_{n-1}}\circ\s{a_n}: E\longmapsto F.$$
It turns out that the $c_q$ only depends on the orientation
$Q$, not on the admissible numbering chosen.
Indeed, if $(a_1, \cdots, a_n)$ and
$(b_1, \cdots, b_n)$ are two admissible numberings of  $Q$, then there
exists an $i$ with $1\leq i\leq n$ such that $b_1=a_i$, because $b_1$ is a sink, there exists no arrow $a_j\to a_i$ with $j<i$ and, because it is easily seen that reflections corresponding to non-neighbours commute, we have
$\s{a_j}\s{a_i}=\s{a_i}\s{a_j}$ for all $j<i$.
The numbering $(a_i, a_1, \cdots, a_{i-1}, a_{i+1}, \cdots, a_n)$ is admissible and an obvious induction implies that
$$\s{a_1}\cdots \s{a_n}=\s{b_1}\cdots\s{b_n}.$$
Thus, $c_q$ is referred to {\bf the $q$-Coxeter transformation of $Q$}
(corresponding to the given admissible numbering).

Let $(a_1, \cdots, a_n)$ be an admissible numbering of  $Q$ and
$S_{a_i}$ the matrix of $\s{a_i}$ in the canonical $\Z$-basis of
%the $\R$-vector space
$E=\R^n$. The matrix
$$\Phi_q=S_{a_1}\cdots S_{a_{n-1}}S_{{a_n}}$$
is called the {\bf $q$-Coxeter matrix of $Q$}.

\begin{ex} For the underline graph $\G$:
$$G: \xymatrix@C=0.5cm{
  1 \ar@{-}[rr] && 2 \ar@{-}[rr] && 3}
 $$
 By definition of $\s{i}$, we have
\begin{eqnarray*}
&&\s{1}(\e_1)=-\e_1, \s{1}(\e_2)=\e_2+q\e_1,
 \s{1}(\e_3)=\e_3,\\
 &&\s{2}(\e_1)=\e_1+q\e_2, \s{2}(\e_2)=-\e_2,
 \s{2}(\e_3)=\e_3+q\e_2.
\end{eqnarray*}
 and
$$\s{3}(\e_1)=\e_1, \s{3}(\e_2)=\e_2+q\e_3,
 \s{3}(\e_3)=-\e_3.
 $$
The corresponding matrices of $\s{i}$ are as follows
 $$
 S_1=\left(
       \begin{array}{ccc}
         -1 & q & 0 \\
         0 & 1 & 0 \\
         0 & 0 & 1 \\
       \end{array}
     \right), \quad
     S_2=\left(
       \begin{array}{ccc}
         1 & 0 & 0 \\
        q & -1 & q \\
         0 & 0 & 1 \\
       \end{array}
     \right), \quad
  S_3=\left(
       \begin{array}{ccc}
         1 & 0 & 0 \\
         0 & 1 & 0 \\
         0 & q & -1 \\
       \end{array}
     \right).
 $$

Given an orientation $Q$ of \ $\G:$
$$\xymatrix@C=0.5cm{
 1\ar@{<-}[rr] && 2 \ar[rr] && 3.}
 $$
The admissible numbering is
$1, 3, 2$,  and the $q$-Coxeter matrix of $Q$ is

$$\Phi_q=S_1S_3S_2=\left(
            \begin{array}{ccc}
              q^2-1 & -q & q^2 \\
              q & -1 & q \\
              q^2 & -q & q^2-1 \\
            \end{array}
          \right).
$$
\end{ex}

The first main result is as follows.
\begin{thm}  \label{thm1} For any acyclic  orientation $Q$ of \ $\G$,
$C_q^T=-\Phi_qC_q,$ or equivalently $\Phi_q=-C_q^TC_q^{-1}.$
\end{thm}
\begin{proof}  By Proposition \ref{prop2-7}, $C_q$ is invertible.

Without loss of generality,
we can assume that
$(1, 2,\cdots, n)$ is an admissible numbering of $Q$.
One knows that $q$-Cartan matrix $C_q$ of $Q$ is
$$C_q=\left(
    \begin{array}{ccccc}
      1 & 0 & 0 & \cdots & 0 \\
      c_{21}(q) & 1 & 0 & \cdots & 0 \\
      c_{31}(q) &c_{32}(q)  & 1 & \cdots & 0 \\
      \vdots &\vdots & \vdots &\ddots & \vdots\\
      c_{n1}(q) & c_{n2}(q) & c_{n3}(q) & \cdots & 1 \\
    \end{array}
  \right) \hbox{ and }
  S_1=\left(
    \begin{array}{ccccc}
      -1 & a_{21}q &   a_{31}q & \cdots & a_{n1}q \\
       0& 1 & 0 & \cdots & 0 \\
    0 &0  & 1 & \cdots & 0 \\
      \vdots &\vdots & \vdots &\ddots & \vdots\\
       &0 & 0 & \cdots & 1 \\
    \end{array}
  \right).
$$

To show that
$C_q^T =-\Phi_q C_q=-S_1\cdots S_n C_q$,
it is sufficient to show
$$C_q^TS_n^T\cdots S_1^T=-C_q$$
or equivalent
$$C_q S_1^T\cdots S_k^T=-C_q^T S_n^T\cdots S_{k+1}^T
=\left(-C_k^T| C_{n-k}\right),$$
where $C_k^T$ (or $C_{n-k}$) is the matrix formed by
the $k$ first columns of $C_q^T$
(or of the $(n-k)$ last columns of $C_q$.

We show that by induction on $k$.

Recall that
$$c_{ij}(q)=\sum_n \dim \left (\e_i(\kk Q)\e_j\right)_nq^n$$
is the $(i, j)$-coefficient of $C_q$.
Moreover, let $b_{ij}$ be the number of arrows from $i$ to
$j$. It is easy to see that
\begin{enumerate}
  \item $b_{ij}=0$ and $b_{ji}=a_{ji}$, for $i\leq j$;
  \item $c_{i+1, i}(q)=b_{i+1, i}q$;
  \item $c_{ii}(q)=1$ for each $i\in Q_0$;
  \item $c_{ij}(q)=q\sum\limits_{j\leq k<i} c_{ik}(q)b_{kj}=q\sum\limits_{j<k\leq i} b_{ik}c_{kj}(q)$,
  for $j<i$.
\end{enumerate}
For $k=1$, we then have
\begin{eqnarray*}
C_qS_1^T&=&\left(
    \begin{array}{ccccc}
      1 & 0 & 0 & \cdots & 0 \\
      c_{21}(q) & 1 & 0 & \cdots & 0 \\
      c_{31}(q) &c_{32}(q)  & 1 & \cdots & 0 \\
      \vdots &\vdots & \vdots &\ddots & \vdots\\
      c_{n1}(q) & c_{n2}(q) & c_{n3}(q) & \cdots & 1 \\
    \end{array}
  \right)\left(
    \begin{array}{ccccc}
      -1 & 0 & 0 & \cdots & 0 \\
      b_{21}q & 1 & 0 & \cdots & 0 \\
      b_{31}q &0  & 1 & \cdots & 0 \\
      \vdots &\vdots & \vdots &\ddots & \vdots\\
      b_{n1}q &0 & 0 & \cdots & 1 \\
    \end{array}
  \right)\\
  &=&\left(
    \begin{array}{ccccc}
      -1 & 0 & 0 & \cdots & 0 \\
      & & & &\\
      -c_{21}(q)+b_{21}q & 1 & 0 & \cdots & 0 \\
      -c_{31}(q)+b_{21}c_{32}(q)q+b_{31}q &c_{32}(q)  & 1 & \cdots & 0 \\
      \vdots &\vdots & \vdots &\ddots & \vdots\\
      -c_{n1}(q)+q\sum\limits_{2\leq i\leq n}b_{i1}c_{ni}(q) &c_{n2}(q) & c_{n3}(q) & \cdots & 1 \\
    \end{array}
  \right).
\end{eqnarray*}
By (2), (3) and (4), we get
$$-c_{21}(q)+b_{21}q=0, \ \cdots, \
-c_{n1}(q)+q\sum\limits_{2\leq i\leq n}b_{i1}c_{ni}(q)=0.$$
Hence
$$C_q S_1^T=\left(-C_1^T| C_{n-1}\right).$$

Now, we assume that the result holds for $k-1$. Hence
$$C_q S_1^T\cdots S_{k-1}^T=\left(-C_{k-1}^T| C_{n-k+1}\right).$$
Then for $k$, we have
\begin{eqnarray*}
&&C_qS_1^T\cdots S_k^T=\left(-C_{k-1}^T| C_{n-k+1}\right)S_k^T\\
&=&\left(
     \begin{array}{ccccccccc}
       -1 & -c_{21}(q) & \cdots & -c_{k-2,1}(q) & -c_{k-1,1}(q) & 0 & 0 & \cdots & 0 \\
       0 & -1 & \cdot & -c_{k-2,2}(q) & -c_{k-1,2}(q) & 0 & 0 & \cdots & 0 \\
       \vdots & \vdots & \ddots & \vdots & \vdots & \vdots & \vdots & \vdots & \vdots \\
       0 & 0 & \cdots & -1 & -c_{k-1, k-2}(q) & 0 &0 & \cdots & 0 \\
       & & & & & & & &\\
       0 & 0 & \cdots & 0 & -1 & 0 & 0 & \cdots & 0 \\
       &  & & & & & & &\\
       0 & 0 & \cdots & 0 & 0 & 1 & 0 & \cdots & 0 \\
       0 & 0 & \cdots & 0 & 0 & c_{k+1, k}(q) & 1 & \cdots & 0 \\
       \vdots & \vdots & \vdots & \vdots & \vdots & \vdots & \vdots & \ddots & \vdots \\
       0 & 0 & \cdots & 0 & 0 & c_{n,k}(q) & c_{n,k+1}(q) & \cdots & 1 \\
     \end{array}
   \right)
\end{eqnarray*}
\begin{eqnarray*}
   &\times &
   \left(
     \begin{array}{cccccccc}
       1 & 0& \cdots & 0 & b_{k1}q & 0 & \cdots & 0 \\
       0 & 1 & \cdots & 0 & b_{k2}q & 0 & \cdots & 0 \\
       \vdots & \vdots & \ddots & \vdots & \vdots  &  &  & \vdots \\
       0 & 0 & \cdots & 1 & b_{k, k-1}q & 0 & \cdots & 0 \\
       0 & 0 & \cdots & 0 & -1 & 0 & \cdots & 0 \\
       0 & 0 & \cdots & 0 & b_{k+1,k}q & 1 & \cdots & 0 \\
      \vdots & \vdots &  & \vdots & \vdots & \vdots & \ddots & \vdots \\
       0 & 0 & \cdots & 0 &b_{n,k}\,q & 0 & \cdots & 1 \\
     \end{array}
   \right)
   =\left(
             \begin{array}{ccc}
               -C_{k-1}^T & \left|
                              \begin{array}{c}
                                -q\sum\limits_{1\leq i\leq k-1}b_{ki}c_{i1}(q) \\
                                -q\sum\limits_{2\leq i\leq k-1}b_{ki}c_{i2}(q) \\
                                \vdots \\
                                -q\sum\limits_{k-2\leq i\leq k-1}b_{ki}c_{i,k-2}(q) \\
                                -qb_{k, k-1} \\
                                -1 \\
                                -c_{k+1, k}(q)+qb_{k+1, k} \\
                                \vdots \\
                               -c_{n,k}(q)+q\sum\limits_{j=k+1}^n c_{n, j}(q)b_{j, k} \\
                              \end{array}
                            \right|
                & C_{n-k} \\
             \end{array}
           \right).
\end{eqnarray*}
The conclusion follows from (2), (3) and (4).

The proof is finished.
\end{proof}

\section{$q$-Cartan (resp. Coxeter) matrix  of $\sigma_iQ$}

Let $Q$ be an orientation of $\G$, which is a finite, connected, and acyclic quiver
and $n=|Q_0|$. For every  $a\in Q_0$, a new quiver
$$\sigma_aQ=(Q'_0, Q'_1, s', t')$$
is defined as follows:
\begin{enumerate}
  \item[(i)] if $s(\a)\ne a$ and $t(\a)\ne a$, then
  $s'(\a')=s(\a)$ and $t'(\a')=t(\a)$;
  \item[(ii)] if $s(\a)=a$ or $t(\a)=a$, then
  $s'(\a')=t(\a)$ and $t'(\a')=s(\a).$
\end{enumerate}

An admissible sequence of sinks in a quiver $Q$ is defined to be a total
ordering $(a_1, \cdots, a_n)$ of all the points in $Q$ such that
\begin{enumerate}
  \item $a_1$ is a sink in $Q$; and
  \item $a_i$ is a sink in $\sigma_{a_{i-1}}\cdots \sigma_{a_1}Q$, for
  every $2\leq i\leq n.$
\end{enumerate}

If $i$ is a sink of $Q$, the $q$-Cartan matrix of $\sigma_iQ$ is
denoted by $\left(\sigma_iC\right)_q$ and
$\left(\sigma_i\Phi\right)_q$ the $q$-Coxeter matrix of $\sigma_i Q$.

\begin{thm} \label{thm2} For a sink $i$ of  an acyclic orientation $Q$ of \ $\G$,  we have
$$ \left(\sigma_i\Phi\right)_q=
S_i\Phi_q S_i, \quad \left(\sigma_iC\right)_q=S_i C_q S_i^T.$$
\end{thm}

\begin{proof}
Assume that $i$ is a sink point of $Q$, then there exists an admissible numbering
$(i=a_1, \cdots, a_n)$  of $Q$.
It is easily seen that if
$(i, a_2, \cdots, a_n)$ is an admissible numbering of $Q$, then
$(a_2,\cdots, a_n, a_1=i)$ is an admissible numbering of $\sigma_{i}Q$.

On the other hand,  $S_{i}^2=1$ and
$$\left(\sigma_i\Phi\right)_q=S_{a_2}\cdots S_{a_n} S_i$$
by the definition of $q$-Coxeter matrix.
We have
$$\left(\sigma_{i}\Phi\right)_q=S_{a_2}\cdots S_{a_n} S_{i}=S_{i}\left(
S_{i}S_{a_2}\cdots S_{a_n}\right) S_{i}=S_i\Phi_q S_i.$$

As the proof of Theorem \ref{thm1}, we can assume that
$(1, 2,\cdots, n)$ is an admissible numbering of $Q$
without loss of generality, where $i=1$.

Under this assumption, one knows that $q$-Cartan matrix $C_q$ of $Q$ is
$$C_q=\left(
    \begin{array}{ccccc}
      1 & 0 & 0 & \cdots & 0 \\
      c_{21}(q) & 1 & 0 & \cdots & 0 \\
      c_{31}(q) &c_{32}(q)  & 1 & \cdots & 0 \\
      \vdots &\vdots & \vdots &\ddots & \vdots\\
      c_{n1}(q) & c_{n2}(q) & c_{n3}(q) & \cdots & 1 \\
    \end{array}
  \right)
$$
and
 $$S_1=\left(
    \begin{array}{ccccc}
      -1 & b_{21}q &   b_{31}q & \cdots & b_{n1}q \\
       0& 1 & 0 & \cdots & 0 \\
    0 &0  & 1 & \cdots & 0 \\
      \vdots &\vdots & \vdots &\ddots & \vdots\\
       &0 & 0 & \cdots & 1 \\
    \end{array}
  \right).$$

Consider the $q$-Cartan matrix of $\sigma_1 Q$:
$1$ is a sink of $Q$, then $1$ is a source of $\sigma_1 Q$.
The entry  of the $(1, j)$ $(j=1, \cdots, n)$  of the first row of  $(\sigma_1C)_q$
is $ q\sum\limits_{j\leq k\leq n} b_{k1}c_{kj}(q),$
where $b_{ij}$ is defined as that
in the proof of Theorem \ref{thm1}.
And the other entries of $(i,j)$ ($i\ne 1$) are equal to $c_{ij}(q)$.

On the other hand,  we see that
\begin{eqnarray*}
C_qS_1^T
  =\left(
    \begin{array}{ccccc}
      -1 & 0 & 0 & \cdots & 0 \\
     % & & & &\\
      0 & 1 & 0 & \cdots & 0 \\
      0 &c_{32}(q)  & 1 & \cdots & 0 \\
      \vdots &\vdots & \vdots &\ddots & \vdots\\
      0 &c_{n2}(q) & c_{n3}(q) & \cdots & 1 \\
    \end{array}
  \right).
\end{eqnarray*}
and
\begin{eqnarray*}
S_1C_qS_1^T&=&\left(
    \begin{array}{ccccc}
      -1 & b_{21}q &   b_{31}q & \cdots & b_{n1}q \\
       0& 1 & 0 & \cdots & 0 \\
    0 &0  & 1 & \cdots & 0 \\
      \vdots &\vdots & \vdots &\ddots & \vdots\\
       &0 & 0 & \cdots & 1 \\
    \end{array}
  \right)\left(
    \begin{array}{ccccc}
      -1 & 0 & 0 & \cdots & 0 \\
      %& & & &\\
      0 & 1 & 0 & \cdots & 0 \\
      0 &c_{32}(q)  & 1 & \cdots & 0 \\
      \vdots &\vdots & \vdots &\ddots & \vdots\\
      0 &c_{n2}(q) & c_{n3}(q) & \cdots & 1 \\
    \end{array}
  \right)\\
  &=&\left(
    \begin{array}{ccccc}
      1 & q\sum\limits_{2\leq k\leq n} b_{k1}c_{k2}(q)  & q\sum\limits_{3\leq k\leq n} b_{k1}c_{k3}(q) & \cdots & b_{n1}q \\
      & & & &\\
      0 & 1 & 0 & \cdots & 0 \\
      0 &c_{32}(q)  & 1 & \cdots & 0 \\
      \vdots &\vdots & \vdots &\ddots & \vdots\\
      0 &c_{n2}(q) & c_{n3}(q) & \cdots & 1 \\
    \end{array}
  \right).
\end{eqnarray*}
This is just the $q$-Cartan matrix $(\sigma_1C)_q$ of $\sigma_1 Q$.

The proof is finished.
\end{proof}

\section{General cases}
In this section, we assume that $(Q, I)$ is a connected  homogeneous bound quiver  of finite global dimension  with $|Q_0|=n$, and $C_q$ is the $q$-Cartan matrix of $(Q, I)$.
\begin{defn}
The $q$-Euler characteristic of $(Q, I)$ is  the $\Z$-bilinear (nonsymmetric) form $\lr{-, -}: \Z^n\times \Z^n\to \Z[q]$
defined by
$$\lr{\x, \y}_q=\x^t C_q^{-1}\y,$$ for $\x, \y\in\Z^n$.
\end{defn}
Now we denote
$C_q^{-1}=\lrb{b_{ij}(q)}_{n\times n}$ and define a symmetric $\Z$-bilinear form as follows:
$$\lrb{\x, \y}_q=\frac{1}{2}\lrb{\y^TC_q^{-1}\x+\x^TC_q^{-1}\y}=\frac{1}{2}\lrb{\lr{\x,\y}_q+\lr{\y, \x}_q}
=\frac{1}{2}\x^TA_q\y,$$
where
$A_q=C_q^{-1}+\lrb{C_q^{-1}}^T=(a_{ij}(q))_{n\times n}$ is a symmetric matrix over $\Z[q]$, and
$$a_{ij}(q)=\left\{
              \begin{array}{ll}
                b_{ij}(q), & \hbox{ if } i>j; \\
                2b_{ii}(q), & \hbox{ if } i=j;\\
                b_{ji}(q), & \hbox{ if } i<j.
              \end{array}
            \right.
$$

Let $\e_1, \e_2, \cdots, \e_n$ be a basis of $\R^n$. We define a {\it $q$-simple reflection $\g{i}$} for $i\in \mb{1, n}$ as
$$\g{i}(\x)=\x-2\lrb{\x, \e_i}_q\e_i$$
for each $i\in Q_0$.

\begin{lem} \label{lem-x} We have
\begin{enumerate}
              \item If $a_{ii}(q)=2$, then $\g{i}^2=1$,
              \item  If $a_{ij}(q)=0$ for some $i, j\in Q_0$, then $\g{i}\g{j}=\g{j}\g{i}$.
  \end{enumerate}
\end{lem}
\begin{proof}
It is easy to see that
$$\g{i}(\e_j)=\e_j-2(\e_i, \e_j)_q\e_i=\e_j-\lrb{\e_i^T A_q\e_j}\e_i
=\e_j-a_{ij}(q) \e_i.$$

If $a_{ii}(q)=2$, we have $\g{i}(\e_i)=-\e_i$, thus
$$\g{i}^2(\e_t)=\g{i}\lrb{\e_t-a_{it}(q)\e_i}=\e_t
-a_{it}(q)\e_i-a_{it}(q)\g{i}(\e_i)=\e_t$$
for all $1\leq t\leq n$. Hence $\g{i}^2=1$.

If $a_{ij}(q)=0$ for some $i$ and $j$, then $a_{ji}(q)=0$,
$\g{i}(\e_j)=\e_i$ and $\g{j}(\e_i)=\e_j$.
Therefore, for $1\leq t\leq n$,
$$\g{i}\g{j}(\e_t)=\g{i}(\e_t-a_{jt}(q)\e_j)
=\e_t-a_{it}(q) \e_i-a_{jt}(q)\e_j$$
is symmetric in $i$ and $j$. In this case,  we get
$\g{i}\g{j}=\g{j}\g{i}$.
\end{proof}

Let
$(a_1, \cdots, a_n)$ be
an admissible numbering of  $Q$, and $S_{a_i}$  the matrix of $\g{a_i}$ for the canonical basis
$\big\{ \e_i | \ i\in \mb{1, n}\big\}$.
The matrix
$$\Phi_q=S_{a_1}\cdots S_{a_{n-1}}S_{{a_n}}$$
is called the {\bf $q$-Coxeter matrix of $(Q, I)$}.

Here, the $q$-Coxeter transformation $\Phi_q$ also only depends on the orientation
$(Q, I)$, not on the admissible numbering chosen.
To see this fact, the following basic fact is used:

{\it if the square matrices $A_1, A_2, A_3$ are invertible, then
$$\left(
    \begin{array}{ccc}
      A_1 & 0 & 0  \\
     0 & A_2 & 0 \\
     B_1 & B_2 & A_3
    \end{array}
  \right)^{-1}=\left(
    \begin{array}{ccc}
       A_1^{-1} & 0 & 0\\
      0 & A_2^{-1} &0\\
      -A_3^{-1}B_1A_1^{-1} & -A_3^{-1}B_2A_2^{-1} &A_3^{-1}
    \end{array}
  \right).$$
}
Now, we assume that $(a_1, \cdots, a_n)$ and
$(b_1, \cdots, b_n)$ are two admissible numberings of  $Q$, then there
exists an $i$ with $1\leq i\leq n$ such that $b_1=a_i$, because $b_1$ is a sink, there exists no paths from $a_k$
to $a_j$ with $j, k\leq i$ and therefore $c_{kj}(q)=0$ for all $j, k\leq i$.  Hence $b_{kj}(q)=0$ for all $j, k\leq i$ by the above fact. It follows that
$\g{a_j}\g{a_i}=\g{a_i}\g{a_j}$ for all $j<i$ by Lemma \ref{lem-x}.
Thus, the numbering $(a_i, a_1, \cdots, a_{i-1}, a_{i+1}, \cdots, a_n)$ is admissible and an obvious induction implies that
$$\g{a_1}\cdots \g{a_n}=\g{b_1}\cdots\g{b_n}.$$

In general, that $i$ and $j$ are not neighbours
in $(Q, I)$ does not imply that $\g{i}\g{j}=\g{j}\g{i}$ for generic parameter $q$.

\begin{ex}  The bound quiver $(Q, I):$
$$\xymatrix@=1.5cm{1&2&3
\ar@<.5ex>"1,1";"1,2"^\alpha
\ar@<-.5ex>"1,1";"1,2"_\beta
\ar@<.ex>"1,2";"1,3"_\delta
}$$
with the relation $I=\{\a\delta-\beta\delta\}$. Then
${\rm gldim}(Q, I)=2$ and
$$C_q = \mbox{{\small $\left( \begin{array}{ccc}
1 & 2q &q^2\\ 0 & 1 &q\\
0 & 0 &1
\end{array} \right)$}}
\hbox{ \
and \
 }
C_q^{-1}= \mbox{{\small $\left( \begin{array}{ccc}
1 & -2q &q^2\\ 0 & 1 &-q\\
0 & 0 &1\\
\end{array} \right)$}}.$$
The matrix of the $q$-reflection $\g{i}, i=1, 2,3 $ are given by
$$S_1=\left(
        \begin{array}{ccc}
          -1 & 2q & -q^2\\
          0 & 1 & 0 \\
          0 & 0 & 1 \\
        \end{array}
      \right), \quad S_2=\left(
        \begin{array}{ccc}
          1 & 0 & 0 \\
          2q & -1 & q \\
          0 & 0 & 1 \\
        \end{array}
      \right), \quad S_3=\left(
        \begin{array}{ccc}
          1 & 0 & 0 \\
          0 & 1 & 0 \\
          -q^2 & q & -1 \\
        \end{array}
      \right).$$
It is obvious that $S_1S_3\ne S_3S_1$ and hence $\g{1}\g{3}\ne \g{3}\g{1}$,
but $1$ and $3$ are not neighbours.

However, it is discovered that
$$S_3S_2S_1=\left(
        \begin{array}{ccc}
          1 & 0 & 0 \\
          2q & -1 & q \\
          q^2 & -q & q^2-1 \\
        \end{array}
      \right)\left(
        \begin{array}{ccc}
          -1 & 2q & -q^2\\
          0 & 1 & 0 \\
          0 & 0 & 1 \\
        \end{array}
      \right)=\left(
                \begin{array}{ccc}
                  -1 & 2q & -q^2 \\
                  -2q & 4q^2-1 & -2q^3+q \\
                  -q^2 & 2q^3-q & -q^4+q^2-1 \\
                \end{array}
              \right)
$$
and
$$C_q^TC_q^{-1}=\left(
     \begin{array}{ccc}
       1 & 0 & 0 \\
       2q & 1 & 0 \\
       q^2 & q & 1 \\
     \end{array}
   \right)\left(
            \begin{array}{ccc}
              1 & -2q & q^2 \\
              0 & 1 & -q \\
              0 & 0 & 1 \\
            \end{array}
          \right)=\left(
                    \begin{array}{ccc}
                      1 & -2q &q^2 \\
                      2q & 1-4q^2 & 2q^3-q \\
                      q^2 & -2q^3+q & q^4-q^2+1 \\
                    \end{array}
                  \right).
$$
Hence $\Phi_q=-C_q^TC_q^{-1}$.
This is not accidental.
\end{ex}

We have the following result.

\begin{thm} \label{thm6-3} Let $C_q$  be the $q$-Cartan matrix of a connected acyclic homogeneous bound quiver
$(Q, I)$ of finite global dimension. Then
$$\Phi_q=-C_q^TC_q^{-1}.$$
\end{thm}
\begin{proof}
 Without loss of generality,
we assume that
$(1, 2,\cdots, n)$ is an admissible numbering of $Q$. One knows that the $q$-Cartan matrix $C_q$ of $Q$ must be of
$$C_q=\left(
    \begin{array}{ccccc}
      1 & 0 & 0 & \cdots & 0 \\
      c_{21}(q) & 1 & 0 & \cdots & 0 \\
      c_{31}(q) &c_{32}(q)  & 1 & \cdots & 0 \\
      \vdots &\vdots & \vdots &\ddots & \vdots\\
      c_{n1}(q) & c_{n2}(q) & c_{n3}(q) & \cdots & 1 \\
    \end{array}
  \right).
$$
It follows that the inverse $C_q^{-1}=(b_{ij}(q))_{n\times n}$ of $C_q$ is also
a lower triangular matrix with $b_{ii}(q)=1$. Hence
\begin{equation}\label{eqn6-1}
 \sum_{j\leq k\leq i} c_{ik}(q)b_{kj}(q)=\sum_{j\leq k\leq i} b_{ik}(q)c_{kj}(q)=\delta_{ij}
\end{equation}
and
$$b_{ij}(q)=0 \ \hbox{ for } i<j, $$
where $c_{ii}(q)=b_{ii}(q)=1$ for all $i=1, \cdots, n$.
Hence we have
\begin{equation}\label{eqn6-2}
 c_{ij}(q)=-\sum_{j<k\leq i} c_{ik}(q)b_{kj}(q)=-\sum_{j\leq k<i} b_{ik}(q)c_{kj}(q)
\end{equation}

Let us determine the matrix $S_i$ of $q$-reflection $\g{i}$.

Recall that
$$\g{i}(\e_j)=\e_j-\lrb{\e_i^T A_q\e_j}\e_i
=\e_j-a_{ij}(q) \e_i,$$
where
$A_q=(a_{ij}(q))_{n\times n}$ is a symmetric matrix over $\Z[q]$, and
$$a_{ij}(q)=\left\{
              \begin{array}{ll}
                b_{ij}(q), & \hbox{ if } i>j; \\
               2b_{ii}(q), & \hbox{ if } i=j;\\
                b_{ji}(q), & \hbox{ if } i<j.
              \end{array}
            \right.
$$
Hence the matrix $S_i$ of $\g{i}$ is
$$S_i=\left(
        \begin{array}{ccccccc}
          1 & 0 & \cdots & 0 &0 & \cdots & 0 \\
          0 & 1 & \cdots & 0 &0 & \cdots & \vdots \\
          \vdots & \vdots &\ddots&\vdots &\vdots &\cdots  & \vdots \\
            0& 0 &\cdots &0 &0 &\cdots &0\\
          -b_{i1}(q) & -b_{i2}(q) & \cdots & 1-2b_{ii}(q) &-b_{i+1,i}(q) & \cdots & -b_{ni}(q)\\
          0 &0  &\cdots &0 &1 &\cdots &\vdots\\
          \vdots & \vdots &\cdots  & \vdots &\vdots&  & \vdots \\
          0& 0 & \cdots &0 &0& \ddots &0\\
          0 & 0 & \cdots & 0 &0 & \cdots & 1 \\
        \end{array}
      \right).
$$

Noting that if $i$ is a sink, then $1-\e_i^TA_q\e_i=1-2b_{ii}(q)=-1$.

To show that
$C_q^T =-\Phi_q C_q=-S_1\cdots S_n C_q$,
it is sufficient to show
$$C_q^TS_n^T\cdots S_1^T=-C_q$$
or equivalent
$$C_q S_1^T\cdots S_k^T=-C_q^T S_n^T\cdots S_{k+1}^T
=\left(-C_k^T| C_{n-k}\right),$$
where $C_k^T$ (or $C_{n-k}$) is the matrix formed by
the $k$ first columns of $C_q^T$
(or of the $(n-k)$ last columns of $C_q$.

We show the fact by induction on $k$.
For $k=1$, we then have
\begin{eqnarray*}
C_qS_1^T&=&\left(
    \begin{array}{ccccc}
      1 & 0 & 0 & \cdots & 0 \\
      c_{21}(q) & 1 & 0 & \cdots & 0 \\
      c_{31}(q) &c_{32}(q)  & 1 & \cdots & 0 \\
      \vdots &\vdots & \vdots &\ddots & \vdots\\
      c_{n1}(q) & c_{n2}(q) & c_{n3}(q) & \cdots & 1 \\
    \end{array}
  \right)\left(
    \begin{array}{ccccc}
      -1 & 0 & 0 & \cdots & 0 \\
      -b_{21}(q) & 1 & 0 & \cdots & 0 \\
      -b_{31}(q) &0  & 1 & \cdots & 0 \\
      \vdots &\vdots & \vdots &\ddots & \vdots\\
      -b_{n1}(q) &0 & 0 & \cdots & 1 \\
    \end{array}
  \right)\\
  &=&\left(
    \begin{array}{ccccc}
      -1 & 0 & 0 & \cdots & 0 \\
      & & & &\\
      -\sum\limits_{1\leq i\leq 2}c_{2i}(q)b_{i1}(q) & 1 & 0 & \cdots & 0 \\
      -\sum\limits_{1\leq i\leq 3}c_{3i}(q)b_{i1}(q) &c_{32}(q)  & 1 & \cdots & 0 \\
      \vdots &\vdots & \vdots &\ddots & \vdots\\
      -\sum\limits_{1\leq i\leq n}c_{ni}(q)b_{i1}(q) &c_{n2}(q) & c_{n3}(q) & \cdots & 1 \\
    \end{array}
  \right).
\end{eqnarray*}
By Eqn. (\ref{eqn6-1}), we get that
$$\sum\limits_{1\leq i\leq 2}c_{2i}(q)b_{i1}(q)=\sum\limits_{1\leq i\leq 3}c_{3i}(q)b_{i1}(q)=\cdots=
\sum\limits_{1\leq i\leq n}c_{ni}(q)b_{i1}(q)=0.$$
Hence
$$C_q S_1^T=\left(-C_1^T| C_{n-1}\right).$$

Now, we assume that the result holds for $k-1$. This means that
$$C_q S_1^T\cdots S_{k-1}^T=\left(-C_{k-1}^T| C_{n-k+1}\right).$$
Then for $k$, we have
\begin{eqnarray*}
&&C_qS_1^T\cdots S_k^T=\left(-C_{k-1}^T| C_{n-k+1}\right)S_k^T\\
&=&\left(
     \begin{array}{ccccccccc}
       -1 & -c_{21}(q) & \cdots & -c_{k-2,1}(q) & -c_{k-1,1}(q) & 0 & 0 & \cdots & 0 \\
       0 & -1 & \cdot & -c_{k-2,2}(q) & -c_{k-1,2}(q) & 0 & 0 & \cdots & 0 \\
       \vdots & \vdots & \ddots & \vdots & \vdots & \vdots & \vdots & \vdots & \vdots \\
       0 & 0 & \cdots & -1 & -c_{k-1, k-2}(q) & 0 &0 & \cdots & 0 \\
       & & & & & & & &\\
       0 & 0 & \cdots & 0 & -1 & 0 & 0 & \cdots & 0 \\
       &  & & & & & & &\\
       0 & 0 & \cdots & 0 & 0 & 1 & 0 & \cdots & 0 \\
       0 & 0 & \cdots & 0 & 0 & c_{k+1, k}(q) & 1 & \cdots & 0 \\
       \vdots & \vdots & \vdots & \vdots & \vdots & \vdots & \vdots & \ddots & \vdots \\
       0 & 0 & \cdots & 0 & 0 & c_{n,k}(q) & c_{n,k+1}(q) & \cdots & 1 \\
     \end{array}
   \right)\\
   &\times &
   \left(
     \begin{array}{cccccccc}
       1 & 0& \cdots & 0 & -b_{k1}(q) & 0 & \cdots & 0 \\
       0 & 1 & \cdots & 0 & -b_{k2}(q) & 0 & \cdots & 0 \\
       \vdots & \vdots & \ddots & \vdots & \vdots  &  &  & \vdots \\
       0 & 0 & \cdots & 1 & -b_{k, k-1}(q) & 0 & \cdots & 0 \\
       0 & 0 & \cdots & 0 & -1 & 0 & \cdots & 0 \\
       0 & 0 & \cdots & 0 & -b_{k+1,k}(q) & 1 & \cdots & 0 \\
      \vdots & \vdots &  & \vdots & \vdots & \vdots & \ddots & \vdots \\
       0 & 0 & \cdots & 0 & -b_{nk}(q) & 0 & \cdots & 1 \\
     \end{array}
   \right)=\left(
             \begin{array}{ccc}
               -C_{k-1}^T & \left|
                              \begin{array}{c}
                                \sum\limits_{1\leq i\leq k-1}b_{ki}(q)c_{i1}(q) \\
                                \sum\limits_{2\leq i\leq k-1}b_{ki}(q)c_{i2}(q) \\
                                \vdots \\
                                \sum\limits_{k-2\leq i\leq k-1}b_{ki}(q)c_{i,k-2}(q) \\
                                b_{k, k-1}(q) \\
                                -1 \\
                                -\sum\limits_{k\leq i\leq k+1}c_{k+1, i}(q) b_{i,k}(q)\\
                                \vdots \\
                               -\sum\limits_{k\leq i\leq n}c_{n, i}(q) b_{i,k}(q) \\
                              \end{array}
                            \right|
                & C_{n-k} \\
             \end{array}
           \right).
\end{eqnarray*}

The conclusion follows from (\ref{eqn6-1}) and (\ref{eqn6-2}).

The proof is finished.
\end{proof}

In general, if $(Q, I)$ is not {\it acyclic} homogeneous bound quiver, we
can not use simple $q$-reflections to define a $q$-Coxeter transformation
 and $q$-Coxeter matrix since we have not an admissible numbering of $(Q, I)$. However, we can still define  {\bf $q$-Coxeter matrix} for any finite homogeneous bound quiver $(Q, I)$ of finite global dimension as follows.

\begin{defn} Let $C_q$ be the $q$-Cartan matrix of an homogeneous bound quiver $(Q, I)$ of finite global dimension. The $q$-Coxeter matrix of $(Q, I)$ is defined by the matrix
$$\Phi_q=-C_q^TC_q^{-1}.$$
\end{defn}

\begin{prop} \begin{enumerate}
               \item $\uldim_q P(i)=-\Phi_q\cdot \uldim_q I(i)$ for $i\in Q_0$.
               \item $\lr{\x, \y}_q=-\lr{\Phi_q\y, \x}_q=\lr{\Phi_q\x, \Phi_q\y}_q$ for
               $\x, \y\in\Z^n$.
             \end{enumerate}
\end{prop}
\begin{proof}
(1)  By applying Proposition \ref{prop2-6}, we get
$$\uldim_q S(i)=\left(C_q^{-1}\right)^T \uldim_q P(i)$$
and
$$\uldim_q I(i)=C_q\cdot\uldim_q S(i)=C_q\left(C_q^{-1}\right)^T \uldim_q P(i).$$
Hence
$$\uldim_q P(i)=-\Phi_q\cdot \uldim_q I(i)$$
for $i\in Q_0.$

(2)
\begin{eqnarray*}
\lr{\x, \y}_q&=&\x^T C_q^{-1} \y=\y^T \left(C_q^{-1}\right)^T \x\\
&=&\y^T\left(C_q^{-1}\right)^T C_q^T C_q^{-1}\x\\
&=&\left(C_q^TC_q^{-1}\y\right)^T C_q^{-1}\x\\
&=&-\lr{\Phi_q\y, \x}_q.
\end{eqnarray*}
This gives the first equality. The second follows on applying the first twice.
\end{proof}

At the end of this paper, we give an example.

\begin{ex}
$$(Q, I): \xymatrix@=1.5cm{1&2
\ar@<-.1ex>"1,1";"1,2"_\alpha
\ar@<-1.8ex>"1,1";"1,2"_\beta
\ar@<-1.4ex>"1,2";"1,1"_\delta
}$$
where $I=\{\alpha\delta, \ \beta\delta\}$.
One sees that  ${\rm gldim}(Q, I)=2$ and the $q$-Cartan matrix of~$(Q, I)$
is
$$C_q=\left(
        \begin{array}{cc}
          1 & 2q \\
          q & 1+2q^2 \\
        \end{array}
      \right)$$
and $\det(C_q)=1$. The inverse $C_q^{-1}$ of $C_q$ is
$$C_q^{-1}=\left(
        \begin{array}{cc}
          1+2q^2 & -2q \\
          -q & 1 \\
        \end{array}
      \right).
$$
The $q$-Coxeter matrix of $(Q, I)$ is
$$\Phi_q=-C_q^TC_q^{-1}=\left(
                          \begin{array}{cc}
                            -1-q^2 & q \\
                            -q(1+2q^2) & 2q^2-1 \\
                          \end{array}
                        \right).
$$
\end{ex}


\begin{thebibliography}{11}

\bibitem{ASS}\label {ASS}  I. Assem, D. Simson, A. Skowro\'{u}ski.
Elements of the Representaion Theory of Associative Algebras. Vol. 1
Cambridge Uinv. Press. Cambridge, New York. 2006.

\bibitem{BT1} C. Bessenrodt, T. Holm. $q$-Cartan Matrices and combinatiorial
invariants of derived categories for skewed-gentle algebras. Pacific J. Math. 229(2007), 25-47.

\bibitem{Ful} K. R. Fuller. The Cartan determinant and global dimension of Artinian rings,
pp. 51–72, in {\it Azumaya algebras, actions, and modules (Bloomington, IN, 1990), edited by D. Haile
and J. Osterburg,} Contemp. Math. 124, Amer. Math. Soc., Providence, RI, 1992.
\bibitem{PM} P. L. M\'{e}liot. Representation Theory
of Symmetric Groups. CRC Press. Taylor \& Francis Group, LLC, 2017.

\bibitem{ZZ} T. Zhao, C. Zhang. $q$-Cartan matrices of
self-injective Nakayama algebras. J. Shandong Univer. (Natural Sci.)
55(10), 2020, 46-51.

\end{thebibliography}
\end{document}